\documentclass[11pt]{amsart}
\usepackage{graphicx} 

\usepackage{amssymb, amsmath, amsthm}

\newtheorem{thm}{Theorem}[section]
\newtheorem{cor}[thm]{Corollary}
\newtheorem{lem}[thm]{Lemma}
\newtheorem{prop}[thm]{Proposition}
\newtheorem{obs}[thm]{Remark}

\numberwithin{equation}{section}

\title[Convergence to initial data of heat problem on $\mathbb{H}_n$]{About the convergence to initial data of the heat problem on the Heisenberg group}

\author{Isolda Eugenia Cardoso}

\address{isolda@fceia.unr.edu.ar, ECEN - FCEIA, Universidad Nacional de Rosario.}


\begin{document}

\begin{abstract}
We study the heat equation on the Heisenberg group and obtain integrability conditions on the initial data that guarantee the existence of the corresponding solutions. We also characterize the weighted Lebesgue spaces for which the solutions converge almost everywhere to the initial data as the time tends to zero. As a further application, we prove weighted estimates for the local maximal operator associated with the heat kernel.
\end{abstract}

\maketitle

\section{Introduction}\label{Sec:intro}

\par Understanding the pointwise convergence of solutions to the heat equation to their initial data, together with its relation to the weighted boundedness of maximal operators, is a classical problem in harmonic analysis that has been extensively studied in the Euclidean setting. The aim of this paper is to obtain analogous results on the Heisenberg group, which is a nonabelian and noncompact Lie group where classical Euclidean techniques do not apply directly. More precisely, we adapt the arguments from \cite{HTV2013} to this setting by exploiting three fundamental features of the Heisenberg group: its homogeneous structure in the sense of Folland and Stein, its sub-Riemannian geometry and the associated Carnot-Carathéodory distance, and its structure as a space of homogeneous type in the sense of Coifman and Weiss.


\par Let us consider the initial value problem for the heat equation associated with the sublaplacian $\mathcal{L}$ on the Heisenberg group $\mathbb{H}_{n}$ ($n\ge 1$). Given $S>0$, we study
\begin{align}\label{heat.equation}
u_{s}(g,s) = & -\mathcal{L}u(g,s), \qquad (g,s)\in\mathbb{H}_{n}\times (0,S), \\ \label{initial.data}
u(g,0) = & f(g), \qquad g\in\mathbb{H}_{n}.
\end{align}

\par Our goal is to obtain integrability conditions on the initial data $f$ such that the function $u(g,s)=e^{-s\mathcal{L}}f(g)$ is well defined on $\mathbb{H}_{n}\times(0,S)$ as an absolutely convergent integral, solves \eqref{heat.equation}, and satisfies $\lim_{s\to 0^{+}}u(g,s)=f(g)$ for almost every $g\in\mathbb{H}_{n}$. We also characterize a class of weights $D_{p}(\mathcal{L})$ for which the above convergence holds for every $f\in L_{v}^{p}(\mathbb{H}_{n})$. Finally, we study the boundedness on $L_{v}^{p}(\mathbb{H}_{n})$ of the corresponding local maximal operator $Q_{a}^{\ast}$.

\par In the Euclidean setting, if we consider the initial value problem for the heat equation associated with the Laplacian $\Delta$ on $\mathbb{R}^{d}$, it is well known that, under suitable conditions on the initial data $f$, the solutions $u(x,s)$ converge to $f(x)$ as $s\to 0^{+}$. Moreover, the solutions can be written in terms of the heat semigroup as $u(x,s)=e^{-s\Delta}f(x)$. Analogous results are also known for the Poisson semigroup.

\par We now mention some antecedents. In \cite{HTV2013}, Hartzstein, Torrea and Viviani characterized the weighted Lebesgue spaces for which the solutions of the heat equation converge a.e. when $s\to 0^{+}$. They also proved the $L^p(v)\to L^p(w)$ boundedness of the local maximal operator for suitable weights $v$ and $w$. Corresponding results for the Poisson equation were also obtained there. In \cite{AFStTo2012}, Abu-Falahah, Stinga and Torrea studied these problems for some heat-diffusion equations associated with the harmonic oscillator and the Ornstein--Uhlenbeck operator. The Hermite and Ornstein--Uhlenbeck settings, for both the heat and Poisson semigroups, were studied by Garrig\'os, Hartzstein, Signes, Torrea and Viviani in \cite{GHSTV2016}, while the Bessel setting was treated by the author in \cite{C2016}.

\par Recently some analogous problems were studied in other spaces different from the Euclidean setting, for example Alvarez-Romero, Barrios and Betancor studied the Laplace operator on homogeneous trees in \cite{ARBaBe2022} and Bruno and Papageorgiou studied several operators on symmetric spaces of noncompact type in \cite{BrPa2023}. These settings have some good properties and sufficient tools have been developed to make them attractive to study.

\par In our work we are motivated by a similar idea: it is interesting to understand how results for PDEs in the Euclidean setting extend to Lie groups in general, and to Carnot or homogeneous groups in particular, since these spaces share the same underlying topological structure. The Heisenberg group $\mathbb{H}_{n}$ is a basic example of such a group: it can be identified with $\mathbb{R}^{2n+1}$ as a manifold, but its group law is noncommutative. On the other hand, a rich theory has been developed on $\mathbb{H}_{n}$, and throughout the paper we will refer to the corresponding literature without aiming to be exhaustive.

\par More precisely, we will be able to reproduce, to a large extent, the results from \cite{HTV2013} for the heat problem by exploiting the rich structure of the Heisenberg group. More specifically, we use the homogeneous structure in the sense of Folland and Stein \cite{FoSt1982}, which provides the diffusion semigroup and the associated heat kernel; the sub-Riemannian structure and its Carnot-Carath\'eodory distance $d$, which together with Li's estimates from \cite{Li2007} provide the precise lower and upper bounds needed to obtain the integrability condition; and the structure of homogeneous type in the sense of Coifman and Weiss \cite{CoWe1977}, which plays a central role in the boundedness of the local maximal operator.

\par We now state our main results:

\begin{thm}\label{Thm:integrability.conditions}
Let $f:\mathbb{H}_{n}\to\mathbb{R}$ be a measurable function such that for every $K>0$,
\[
\int_{|z|\le K} |f(h)| \, dh < \infty.
\]
Assume that for every $s\in(0,S)$,
\begin{equation}\label{condint}
	\int\limits_{\mathbb{H}_{n}} \phi_{s}(g)|f(g)|dg<\infty,
\end{equation}
where, for $g=(z,t)\in\mathbb{H}_{n}$,
\begin{equation*}
	\phi_{s}(g)= \frac{1}{s^{n-1}} e^{-\frac{d(g)^2}{4s}}
	\left(1+\frac{d(g)^{2}}{s}\right)^{n-1}
	\left( 1 + \frac{|z|d(g)}{s} \right)^{-n+\frac{1}{2}},
\end{equation*}
and $d$ denotes the Carnot--Carath\'eodory distance.

Then the heat integral $u(g,s)$ is absolutely convergent and satisfies:
\begin{enumerate}
	\item[(i)] $u(g,s)\in C^{\infty}(\mathbb{H}_{n}\times(0,S))$ and solves the heat equation associated with the Heisenberg sublaplacian;
	\item[(ii)] $\lim\limits_{s\to 0^{+}} u(g,s) = f(g)$ for a.e. $g\in\mathbb{H}_{n}$.
\end{enumerate}

Conversely, if the heat integral of $f$ is finite for every $s\in(0,S)$ and some $g\in\mathbb{H}_{n}$, then $f$ satisfies \eqref{condint}.
\end{thm}

\par From Theorem \eqref{Thm:integrability.conditions} we define a class of weights such that (i) and (ii) hold for every initial data $f$ in a weighted $L^p$ space. More precisely, for $1\le p <\infty$, let $D_{p}(\mathcal{L},S)$ be the class of all weights $v: \mathbb{H}_{n}\to\mathbb{R}^{+}$ such that (i) and (ii) from Theorem \ref{Thm:integrability.conditions} hold for every $f\in L^{p}(\mathbb{H}_{n},v)$. This class can be characterized as follows.

\begin{cor}\label{Cor:weight.characterization}
Let $1\le p<\infty$ and let $p'$ be such that $\frac{1}{p}+\frac{1}{p'}=1$. A weight $v$ belongs to the class $D_{p}(\mathcal{L},S)$ if and only if $v^{-\frac{1}{p}}\phi_{s} \in L^{p'}(\mathbb{H}_{n})$ for every $s\in(0,S)$.
\end{cor}

\par Finally, we prove a boundedness result for the local maximal operator defined by
\begin{align}\label{local.max.op}
Q_{a}^{\ast}f(g)= \sup\limits_{0<s<a} |e^{-s\mathcal{L}}f(g)|.
\end{align}

\begin{thm}\label{Thm:local.maximal.operator.weightedLp.boundedness}
Let $1<p<\infty$. If $v\in D_{p}(\mathcal{L},S)$, then for every $a\in(0,S)$ there exists a weight $u=u_a$ such that
\begin{align}\label{Qast.bd.p}
	Q^{\ast}_a:L^{p}(\mathbb{H}_n,v)\to L^{p}(\mathbb{H}_n,u)
\end{align}
is bounded. Moreover, if $q>p$, then $u$ can be chosen in $D_{q}(\mathcal{L},S)$.

Conversely, if \eqref{Qast.bd.p} holds for some $a\in(0,S)$ and some weight $u=u_a$, then $v\in D_{p}(\mathcal{L},S)$.
\end{thm}

\par In Section \ref{Sec:prelim} we describe our setting, in Section \ref{Sec:bricks} we prove some auxiliary lemmas, and in Section \ref{Sec:proofs.of.thms} we prove the main theorems.

\par Before continuing, let us briefly explain why we restrict our attention to the heat equation and do not consider the Poisson problem. In the Heisenberg setting, the Poisson semigroup is defined through the fractional powers of the sublaplacian via the subordination formula $e^{-s\mathcal{L}^{\frac12}}f$, and its analysis requires techniques that are substantially different from those used in the heat setting. Although several aspects of the Poisson problem are certainly interesting, treating them would considerably increase the scope and length of the paper. For this reason, we focus exclusively on the heat equation.

\section{Preliminaries}\label{Sec:prelim}

\par We work on the Heisenberg group $\mathbb{H}_{n}=\mathbb{C}^{n}\times\mathbb{R}$ endowed with the group law
$$(z,t)(z',t')=\left(z+z',t+t'+2\Im \langle z,\overline{z'} \rangle\right),$$
where $\langle\cdot,\cdot\rangle$ denotes the canonical Hermitian inner product on $\mathbb{C}^n$,
$$\langle z,\overline{z'} \rangle=\sum\limits_{j=1}^{n} z_{j}\overline{z'_{j}}.$$
The Heisenberg group is a real Lie group whose Haar measure coincides with the Lebesgue measure $dzdt$ on $\mathbb{C}^{n}\times\mathbb{R}$. When convenient, we will denote its elements by $g\in\mathbb{H}_{n}$, although for most computations we will use the coordinates $(z,t)$. The identity element is $(0,0)$ and the inverse of $(z,t)$ is $(-z,-t)$.

\par The Lie algebra $\mathfrak{h}_{n}$ has the canonical basis
$$\{X_{1},\dots,X_{n},Y_{1},\dots,Y_{n},T\},$$
where for $j=1,\dots,n$,
$$X_{j}= \frac{\partial}{\partial x_{j}} - 2 y_{j} \frac{\partial}{\partial t},\qquad
Y_{j}= \frac{\partial}{\partial y_{j}} + 2 x_{j} \frac{\partial}{\partial t},
\qquad \text{and}\qquad
T= \frac{\partial}{\partial t}.$$
The Lie bracket satisfies $[X_{j},Y_{j}]=4T$ for $j=1,\dots,n$, and therefore $T$ generates the center of the Lie algebra.

\par The Heisenberg sublaplacian is the second order left invariant differential operator
$$ \mathcal{L} = \sum\limits_{j=1}^{n} X_{j}^{2}+Y_{j}^{2},$$
which in coordinates can be written as
\begin{align*}
	\mathcal{L}=&  \sum\limits_{j=1}^{n} \left( \frac{\partial^{2}}{\partial x_{j}^{2}} +\frac{\partial^{2}}{\partial y_{j}^{2}}   \right)
	+ 4\sum\limits_{j=1}^{n}(x_{j}^{2}+y_{j}^{2})\frac{\partial^{2}}{\partial t^{2}} \\
	& + 4 \sum\limits_{j=1}^{n} \left( y_{j} \frac{\partial}{\partial x_{j}} - x_{j}\frac{\partial}{\partial y_{j}} \right)\frac{\partial}{\partial t} \\
	= & \Delta_{\mathbb{R}^{2n}} +4 |z|^{2}T^{2} + 4 \mathcal{R} T,
\end{align*}
where
$$ \mathcal{R}=  \sum\limits_{j=1}^{n} \left( y_{j} \frac{\partial}{\partial x_{j}} - x_{j}\frac{\partial}{\partial y_{j}} \right).$$
Although $\mathcal{L}$ is not elliptic, it is hypoelliptic by H\"ormander's theorem \cite{Ho1967} (see also \cite{Ka1980}) and satisfies subelliptic estimates \cite{FoKo1972}. For our purposes, $\mathcal{L}$ plays the role of the Laplacian in Euclidean space.

\par The Heisenberg group is an homogeneous Lie group with nonisotropic dilations
$$\delta_{r}(z,t)=(rz,r^2 t), \qquad r>0.$$
The operator $\mathcal{L}$ is, up to a multiplicative constant, the unique left invariant and rotation invariant differential operator homogeneous of degree $2$, in the sense that
$$\mathcal{L}(\delta_{r}f)=r^{2}\delta_{r}(\mathcal{L}f).$$
Moreover, $\mathcal{L}$ is formally self-adjoint and nonnegative. Therefore, the heat equation \eqref{heat.equation} with initial data \eqref{initial.data} generates a diffusion semigroup $e^{-s\mathcal{L}}$, and its solution is given by
$$u(z,t;s)=e^{-s\mathcal{L}}f(z,t).$$
See for example the monograph by Stein \cite{St1970}.

\par A theorem of Hunt \cite{Hu1956} guarantees the existence of a probability measure $\mu_{s}$, absolutely continuous with respect to the Haar measure, such that 
$$ e^{-s\mathcal{L}}f(z,t) = \int\limits_{\mathbb{H}_{n}}
f((z,t)(z',t')^{-1})\, d\mu_{s}(z',t').$$
If $q_{s}(z,t)$ denotes the density of $\mu_s$, then
$$ u(z,t;s) = e^{-s\mathcal{L}}f(z,t)=f\ast q_{s}(z,t),$$
where convolution on $\mathbb{H}_{n}$ is defined by
$$ f\ast q(z,t)= \int\limits_{\mathbb{H}_{n}} f(z',t')q((z,t)(z',t')^{-1})\, dz'dt'.$$

\par The kernel $q_{s}$ is a nonnegative function in $C^{\infty}(\mathbb{H}_{n}\times (0,\infty))$ satisfying
$$ \int_{\mathbb{H}_n} q_s(z,t)\,dzdt=1 $$
and the scaling relation
$$ q_{r^2s}(z,t) = \frac{1}{r^{2n+2}} q_{s}(\delta_{r^{-1}}(z,t)).$$
An explicit formula for $q_s$ can be obtained by taking the inverse Fourier transform in the central variable (see for example \cite{Th2012}):
\begin{equation}\label{explicit.heat.kernel}
	q_{s}(z,t) = \frac{1}{2(4\pi s)^{n+1}} \int\limits_{-\infty}^{\infty} \left( \frac{\lambda}{\sinh \lambda} \right)^{n} e^{-\frac{|z|^{2}}{4s}  \lambda \coth \lambda } e^{i \lambda \frac{t}{s}} d\lambda.
\end{equation}

\par From this formula one can prove that if $f\in L^{p}(\mathbb{H}_{n})$, with $1\le p<\infty$, then $f\ast q_{s}$ converges to $f$ in $L^{p}(\mathbb{H}_{n})$ as $s\to0$.

\par The integral formula defining the heat kernel $q_s$ is often difficult to handle and differs in several respects from the Euclidean case. For instance, $q_{s}(z,t)$ does not exhibit Gaussian decay in the central variable; see \cite{Th2012}, where this phenomenon appears in the study of Hardy's theorem on $\mathbb{H}_{n}$. Since in this work we are mainly interested in the behavior of the heat kernel, we recall some estimates that will be useful later on.

\par There is an extensive literature on heat kernel estimates, global bounds, asymptotic behavior, and related questions on the Heisenberg group. We mention in particular the influential work of Gaveau \cite{Ga1977}, from which much of the subsequent literature originates. Because our arguments require rather precise estimates, we will use the bounds obtained by Q. Li in \cite{Li2007}, later revisited by Q. Li and Zhang in \cite{LiZa2019}. In that work the authors present a detailed study of the heat kernel on isotropic and nonisotropic Heisenberg groups, including a careful description of its asymptotic behavior in several different regimes, together with many historical references on the subject.
 
 \par These estimates involve some geometric properties of the underlying space, which we now briefly recall. The sub-Riemannian structure on $\mathbb{H}_{n}$ induces the Carnot--Carath\'eodory distance $d$: for $g=(z,t)$ and $g'=(z',t')$ in $\mathbb{H}_n$, the quantity $d(g,g')$ is defined as the infimum of the lengths of horizontal curves joining $g$ and $g'$. If we denote by $d(g)$ the distance from $g$ to the origin, then by \cite[(1.19)]{LiZa2019},
 $$ d(z,t)^2 \sim |z|^2+|t|$$
 for every $(z,t)\in\mathbb{H}_{n}$.
 
 \par Finally, let us recall that the Heisenberg group is a space of homogeneous type in the sense of Coifman and Weiss (see \cite[Chapter III]{CoWe1977}). In particular, if $B(g,r)$ denotes the ball centered at $g\in\mathbb{H}_{n}$ with radius $r>0$ with respect to the Carnot-Carath\'eodory distance, then the Haar measure satisfies the doubling condition
 $$|B(g,2r)| \le C |B(g,r)| $$
 for some constant $C>0$ independent of $g$ and $r$. This allows us to use several tools from harmonic analysis on spaces of homogeneous type.
 
\section{Technical Results}\label{Sec:bricks}

\par Let us begin by recalling the heat kernel estimate. From Theorem 1 of \cite{Li2007}, if we define
\begin{equation}\label{candidate}
	\phi_{s}(z,t) = \frac{1}{s^{n-1}} e^{-\frac{d(z,t)^2}{4s}}
	\left(1+\frac{d(z,t)^{2}}{s}\right)^{n-1}
	\left( 1 + \frac{|z|d(z,t)}{s} \right)^{-n+\frac{1}{2}},
\end{equation}
for $s>0$ and $(z,t)\in\mathbb{H}_{n}$, then there exists a constant $A>1$ such that
\begin{equation}\label{est.phi.q}
	A^{-1} \phi_{s}(z,t) \le q_{s}(z,t) \le A \phi_{s}(z,t).
\end{equation}
 Observe that the function $\phi_s:\mathbb{H}_{n}\to\mathbb{R}$ is nonnegative and continuous with respect to the metric $d$.

\par Let us now establish some estimates for later use. Let $M>1$ and $g_{0}=(z_{0},t_{0})\in\mathbb{H}_{n}$. If $h=(z,t)\in\mathbb{H}_{n}$ satisfies $d(h)>Md(g_{0})$, then, since $d$ is a metric and satisfies the triangle inequality,
\begin{equation}\label{M.d.estimates}
	\left(\frac{M-1}{M} \right) d(h) 
		\le d(g_{0}h^{-1})
		\le \left( \frac{M+1}{M} \right) d(h).
\end{equation}
Similarly, whenever $|z|> M |z_{0}|$,
\begin{equation}\label{M.euclid.estimates}
	\left(\frac{M-1}{M} \right) |z|
		\le |z_{0}-z|
		\le \left( \frac{M+1}{M} \right) |z|.
\end{equation}

\begin{lem}\label{lem:bilateral.estimation}
Let $M>1$, $s>0$, and $g=(z_{0},t_{0})\in\mathbb{H}_{n}$. If $h=(z,t)\in \mathbb{H}_{n}$ satisfies
$$	d(h)>M d(g) \qquad \text{and} \qquad |z|>M|z_{0}|,$$
then the following bilateral estimate holds:
\begin{equation}\label{bilateral.estimation}
	A_1 \phi_{\sigma_1}(h)
		\le \phi_{s}(gh^{-1})
		\le A_2 \phi_{\sigma_2}(h),
\end{equation}
where
$$m_1:=\left(\frac{M}{M+1}\right)^2 < 1 < m_2:=\left(\frac{M}{M-1}\right)^2,$$
and
$$ C_M:=\frac{m_2}{m_1} = \left(\frac{M+1}{M-1}\right)^2.	$$
Moreover, if
$$\sigma_1:=m_1 s\qquad \text{and} \qquad \sigma_2:=m_2 s,$$
then
$$ \sigma_1<s<\sigma_2, $$
and
$$A_1:=\left(\frac{m_1}{C_M}\right)^{n-1}, \qquad 	A_2:=\left(m_2 C_M\right)^{n-1}.$$
\end{lem}

\begin{proof}
\par The estimate follows from \eqref{M.d.estimates} and \eqref{M.euclid.estimates} by comparing the different factors appearing in the definition of $\phi_s$.

\par Indeed, from \eqref{M.d.estimates} we obtain
$$ m_1 d(h)^2 \le d(gh^{-1})^2 \le m_2 d(h)^2, $$
which immediately gives
\begin{equation}\label{exp.estimate.full}
	e^{-\frac{d(h)^2}{4\sigma_1}}
	\le e^{-\frac{d(gh^{-1})^2}{4s}}
	\le e^{-\frac{d(h)^2}{4\sigma_2}}.
\end{equation}

\par Combining \eqref{M.d.estimates} and \eqref{M.euclid.estimates}, we also have
\begin{equation}\label{-n+1/2.estimate.full}
	\left ( 1 +\frac{|z|d(h)}{\sigma_1}\right)^{-n+\frac{1}{2}}
			\le \left ( 1 + \frac{|z-z_0|d(gh^{-1})}{s}\right)^{-n+\frac{1}{2}}
			\le \left ( 1 + \frac{|z|d(h)}{\sigma_2}\right)^{-n+\frac{1}{2}},
\end{equation}
and
\begin{equation}\label{n-1.estimate.full}
	\frac{1}{{C_M}^{n-1}} \left( 1 + \frac{d^2(h)}{\sigma_1} \right)^{n-1}
	\le \left( 1 + \frac{d^2(gh^{-1})}{s} \right)^{n-1}
	\le {C_M}^{n-1} \left( 1 + \frac{d^2(h)}{\sigma_2} \right)^{n-1}.
\end{equation}

\par The conclusion now follows by multiplying the previous estimates factor by factor in the definition of $\phi_s$.
\end{proof}

\par Since the function
$$\left ( 1 + \frac{|z|d(h)}{\sigma_1}\right)^{-n+\frac{1}{2}} $$
is uniformly bounded by $1$, we can also establish an upper estimate in the case where $|z|$ remains bounded while $d(h)$ is large.

\begin{cor}\label{cor:upper.estimation.on.B_g}
With the notation of Lemma \eqref{lem:bilateral.estimation}, if we assume instead that $|z|\le M|z_{0}|$, then
\begin{equation}\label{upper.estimation}
	\phi_s(g_{0}h^{-1})
	\le
	A_2 \phi_{\sigma_2}(h)\varphi_{\sigma_2}(g_{0}),
\end{equation}
where
$$\varphi_s(g_{0}) =M^2\left(1 + \frac{|z_0|d(g_{0})}{s}\right)^{n-\frac{1}{2}}.$$
\end{cor}

\begin{proof}
\par By Lemma \eqref{lem:bilateral.estimation}, the only factor requiring a different treatment is
$$ \left ( 1 + \frac{|z-z_0|d(g_{0}h^{-1})}{s}\right)^{-n+\frac{1}{2}}, $$
which is bounded above by $1$.

\par In order to recover the factor appearing in $\phi_{\sigma_2}(h)$, we write
\begin{align*}
	1 	=& 	\left( 1 + \frac{M|z|d(g_{0})}{\sigma_2}\right)^{-n+\frac{1}{2}} 	\left( 1 + \frac{M|z|d(g_{0})}{\sigma_2}\right)^{n-\frac{1}{2}} \\
	\le & 	\left( 1 + \frac{|z|d(h)}{\sigma_2}\right)^{-n+\frac{1}{2}} 	\left( 1 + \frac{M^2|z_0|d(g_{0})}{\sigma_2}\right)^{n-\frac{1}{2}},
\end{align*}
since $d(h)>Md(g_{0})$ and $|z|\le M|z_{0}|$. The conclusion follows.
\end{proof}

For the next results, we consider the function $\phi_{s}$ as an integrability factor to characterize the solutions for the IVP \eqref{heat.equation}-\eqref{initial.data} associated with an initial data $f$. As neither the heat kernel nor the function $\phi_{s}$ decays as a Gaussian in the $t$ variable, we need a local integrability assumption on the initial data.

\begin{prop}\label{Prop:integrability.conditions}
Let $f:\mathbb{H}_{n}\to\mathbb{R}$ be a measurable function such that for every $K>0$, the integral $\int_{|z|\le K, |t|\le K} |f(h)| dh$ is finite. Let $S>0$ be fixed. Then the following statements are equivalent:
\begin{enumerate}
\item[(i)] For every $s\in (0,S)$ and every $g\in \mathbb{H}_{n}$,
\begin{equation*}
	\int\limits_{\mathbb{H}_{n}} q_{s}(gh^{-1}) |f(h)| dh < \infty.
\end{equation*}
\item[(ii)] For every $s\in (0,S)$ and some $g_{s}\in \mathbb{H}_{n}$,
\begin{equation*}
	\int\limits_{\mathbb{H}_{n}} q_{s}(g_{s}h^{-1}) |f(h)| dh < \infty.
\end{equation*}
\item[(iii)] For every $s\in (0,S)$,
\begin{equation*}
	\int\limits_{\mathbb{H}_{n}} \phi_{s}(h) |f(h)| dh < \infty,
\end{equation*}
where the integrability factor $\phi_{s}$ is the function defined in \eqref{candidate}.
\end{enumerate}
\end{prop}

\begin{proof}
The implication (i) $\implies$ (ii) is immediate. Let us show that (ii) implies (iii). Fix $s\in (0,S)$. Since $s < S$, we can choose $M>1$ sufficiently large so that $\sigma:=s\left(\frac{M+1}{M}\right)^2 < S$. This choice implies $s=\sigma \left( \frac{M}{M+1} \right)^2$, allowing us to apply Lemma \ref{lem:bilateral.estimation}. Specifically, for $d(h)>Md(g)$ and $|z|> M |z_0|$, we have
\begin{align}\label{est.I.1.1}
	\phi_{s}(h) \le \frac{1}{A_1} \phi_{\sigma}(gh^{-1}).
\end{align}
By hypothesis (ii) applied to this $\sigma$, there exists a point $g_{\sigma}=(z_{\sigma},t_{\sigma})\in\mathbb{H}_{n}$ such that
\begin{align}\label{est.I.1.2}
\int\limits_{\mathbb{H}_{n}} q_{\sigma}(g_{\sigma}h^{-1}) |f(h)| dh < \infty.
\end{align}
We split the integral into two regions as follows:
\begin{align*}
	\int\limits_{\mathbb{H}_{n}} \phi_{s}(h) |f(h)| dh &=  \int\limits_{\mathcal{A}_{g_{\sigma}}} \phi_{s}(h) |f(h)| dh + \int\limits_{(\mathcal{A}_{g_{\sigma}})^{C} } \phi_{s}(h) |f(h)| dh =: I_1 + I_2,
\end{align*}
where $\mathcal{A}_{g_{\sigma}} = \{ h=(z,t)\in\mathbb{H}_{n}: d(h)>Md(g_{\sigma}), |z|> M |z_{\sigma}| \}$.

 Combining \eqref{est.I.1.1}, \eqref{est.I.1.2}, and the equivalence \eqref{est.phi.q}, it immediately follows that $I_1$ is finite.
	
To analyze $I_2$, we split the complement $(\mathcal{A}_{g_{\sigma}})^{C}$ into the compact region $\mathcal{C}_{g_{\sigma}} = \{ h\in\mathbb{H}_{n} : d(h)\le M d(g_{\sigma}) \}$ and the unbounded region $\mathcal{B}_{g_{\sigma}}=\{ h=(z,t)\in\mathbb{H}_{n} : d(h)>Md(g_{\sigma}), |z|\le M |z_{\sigma}| \}$. On the compact set $\mathcal{C}_{g_{\sigma}}$, the continuity of $\phi_{s}$ together with the local integrability of $f$ guarantees that $\int_{\mathcal{C}_{g_{\sigma}}} \phi_s |f| dh < \infty$. 
	
The region $\mathcal{B}_{g_{\sigma}}$ presents an extra difficulty since $|z|$ is bounded but $t$ (and thus also $d(h)$) can be arbitrarily large. However, for $h \in \mathcal{B}_{g_{\sigma}}$, the definition of $\phi_s$ yields the estimate
 $$ \phi_s(h) \le C e^{-\frac{d(h)^2}{4s}} d(h)^{2n-2} \le C' e^{-\frac{|t|}{4s}} (1+|t|)^{n-1},$$
where we used that $d(h)^2 = \sqrt{|z|^4+t^2} \ge |t|$ and that $|z|$ is bounded by $M|z_\sigma|$. Splitting the vertical direction into intervals $J_k = \{t \in \mathbb{R} : k \le |t| < k+1\}$, the integral over $\mathcal{B}_{g_{\sigma}}$ can be bounded by a convergent series:
$$ \int\limits_{\mathcal{B}_{g_{\sigma}}} \phi_{s}(h)|f(h)| dh \le \sum_{k=0}^\infty C_k \int_{|z|\le M|z_\sigma|, |t|\in J_k} |f(z,t)| dzdt,$$
where $C_k \sim e^{-k/4s} k^{n-1}$. From the assumption on $f$ and the convergence of the series, we obtain that $I_2$ is finite, thus establishing (iii).

To complete the proof, it remains to show that (iii) implies (i). Fix $g \in \mathbb{H}_n$ and $s \in (0,S)$. We can choose $\sigma \in (s, S)$ and $M>1$ such that, by the alternative inequality in Lemma \eqref{bilateral.estimation}, we have $q_s(gh^{-1}) \le C \phi_\sigma(h)$ for $h$ far from $g$. Then, we can split the integral of $q_s(gh^{-1})|f(h)|$ into a compact set and an unbounded region exactly as before. The integral over the compact set is finite by the local integrability of $f$, while the integral over the remaining region is bounded by $\int_{\mathbb{H}_n} \phi_\sigma |f| dh$, which is finite by (iii). This shows that (i) holds and completes the proof.
\end{proof}

\begin{obs}\label{rmk:well-def}
 Under the conditions of Proposition \ref{Prop:integrability.conditions}, for each $s\in(0,S)$, the convolution $$(q_{s}\ast f)(g)=\int\limits_{\mathbb{H}_{n}}q_s(gh^{-1})f(h)dh$$ is well defined as an absolutely convergent integral for $g\in\mathbb{H}_{n}$.
\end{obs}

The next proposition justifies that the convolution from Remark \ref{rmk:well-def} is a $C^\infty$ function. Even though we know that the heat kernel is $C^\infty$, the initial data $f$ does not need to belong to $L^{1}(\mathbb{H}_{n})$ (we only ask $f$ to satisfy the hypothesis of the Proposition), hence the need for a proof. For this, we resort again to the work of Q. Li \cite{Li2007}. More precisely, let us recall Theorem 2 about the derivatives of the heat kernel: for $\alpha_{j} \in\mathbb{N}$, $U_{j}\in \{X_{1},\dots,X_{n},Y_{1},\dots,Y_{n}\}$ with $1\le j\le k$, $k\in\mathbb{N}_{0}$, there exists a constant $C>0$ such that
\begin{align}\label{derivatives.estimates}
	|U_{1}^{\alpha_{1}}\cdots U_{k}^{\alpha_k}q_{s}(g)| \le & C \left(1+\frac{|z|}{\sqrt{s}}\right)^{|\alpha|}q_{s}(g),
\end{align}
for any $g=(z,t)\in\mathbb{H}_{n}$, where $|\alpha| = \alpha_{1}+\dots+\alpha_{k}$.

\begin{prop}\label{Prop:conv.is.C.infty} 
	If a measurable function $f$ satisfies the hypothesis from Proposition \ref{Prop:integrability.conditions} as well as any of its conditions, then
	\begin{align*}
		u(g,s) = & \int\limits_{\mathbb{H}_{n}} q_{s}(gh^{-1})f(h)dh \in C^{\infty}(\mathbb{H}_{n}\times (0,S)).
	\end{align*}
\end{prop}

\begin{proof}
We need to prove that 
\begin{align}\label{derivatives.conv} 
	\int\limits_{\mathbb{H}_{n}} |Dq_{s}(gh^{-1}) f(h)| dh < \infty 
\end{align} 
for every $s\in(0,S)$, every $g\in\mathbb{H}_{n}$ and every differential operator $D$.
	
Any differential operator $D$ can be regarded as a polynomial in the left-invariant horizontal differential operators $\{X_{1},\dots,X_{n},Y_{1},\dots,Y_{n}\}$ with $C^\infty$ coefficients. Indeed, since $q_s$ satisfies the heat equation, any time derivative $\frac{\partial}{\partial s}$ can be replaced by the sub-Laplacian $\mathcal{L}$. Also, the central operator $T$ can be written in terms of the horizontal fields via the commutation relation $T = \frac{1}{4}(X_1 Y_1 - Y_1 X_1)$. Since any standard coordinate derivative can be expressed using horizontal fields and $T$ (for instance, $\frac{\partial}{\partial x_{1}} = X_{1} + 2 y_{1} T$), linearity allows us to restrict our attention to monomials in the horizontal fields.
	
For $D=U_{1}^{\alpha_{1}}\cdots U_{k}^{\alpha_k}$, estimates \eqref{derivatives.estimates} and \eqref{est.phi.q} yield
\begin{align}\label{deriv.ineq}
	|Dq_{s}(h)| \le & C \left(1+\frac{|z|}{\sqrt{s}}\right)^{|\alpha|}|q_{s}(h)|  \le  C \left(1+\frac{|z|}{\sqrt{s}}\right)^{|\alpha|}|\phi_{s}(h)| 
\end{align}
for every $h = (z,t) \in\mathbb{H}_{n}$ and $s\in (0,S)$.
	
Given $s\in (0,S)$, choose $M>1$ and $\sigma>0$ such that $s = \left( \frac{M-1}{M} \right)^{2}\sigma < S$, and let $g_{\sigma} = (z_{\sigma},t_\sigma)$. We split the integral \eqref{derivatives.conv} over the three sets $\mathcal{A}_{g_{\sigma}}$, $\mathcal{B}_{g_{\sigma}}$ and $\mathcal{C}_{g_{\sigma}}$ defined in the proof of Proposition \ref{Prop:integrability.conditions}.
	
For $h = (z,t) \in \mathcal{A}_{g_{\sigma}}$, the estimates \eqref{M.euclid.estimates} and \eqref{bilateral.estimation} hold, which leads to
$$\int\limits_{\mathcal{A}_{g_{\sigma}}} |Dq_{s}(g_{\sigma}h^{-1}) f(h)| dh \le C(M,|\alpha|) \int\limits_{\mathcal{A}_{g_{\sigma}}} \left( 1 + \frac{|z|}{\sqrt{\sigma}} \right)^{|\alpha|} |\phi_{\sigma}(h)| |f(h)| dh.$$
The auxiliary function $\overline{f}(h) = \left( 1 + \frac{|z|}{\sqrt{\sigma}} \right)^{|\alpha|} f(h)$ is easily seen to satisfy the conditions of Proposition \ref{Prop:integrability.conditions}, meaning this integral is finite. 

Over the sets $\mathcal{B}_{g_{\sigma}}$ and $\mathcal{C}_{g_{\sigma}}$, the factor $\left( 1 + \frac{|z|}{\sqrt{\sigma}} \right)^{|\alpha|}$ is bounded by a constant since $|z|$ is bounded. On the compact set $\mathcal{C}_{g_{\sigma}}$, $\phi_{s}$ is bounded and $f$ is locally integrable, which ensures the integral is finite. Over $\mathcal{B}_{g_{\sigma}}$, the polynomial factor is bounded, and the integral converges by the same argument of decay in the $t$ variable used in Proposition \ref{Prop:integrability.conditions}. Finally, it follows that \eqref{derivatives.conv} holds.
\end{proof}

\begin{obs}\label{obs.low.cont}
	A nonnegative function $f$ satisfying $e^{-s\mathcal{L}}f(g)<\infty$ for some $g\in\mathbb{H}_{n}$ and some $s>0$ is locally integrable.
\end{obs}

Indeed, if $\mathcal{C}$ is a compact set in $\mathbb{H}_{n}$, the function $h\to q_{s}(gh^{-1})$ is continuous and strictly positive everywhere. Hence, it is bounded from below on $\mathcal{C}$ by a positive constant $c=c(g,s,\mathcal{C})>0$, and we have
\begin{align*}
	\int\limits_{\mathcal{C}} f(h) dh \le \frac{1}{c} \int\limits_{\mathcal{C}} q_{s}(gh^{-1}) f(h) dh \le \frac{1}{c} \int\limits_{\mathbb{H}_{n}} q_{s}(gh^{-1}) f(h) dh < \infty.
\end{align*}


The following proposition guarantees that, under our minimal integrability conditions, the solution recovers the initial data $f$ almost everywhere as time approaches zero.

\begin{prop}\label{Prop:limit.holds.ae} 
If a measurable function $f$ satisfies the hypothesis from Proposition \ref{Prop:integrability.conditions} as well as any of its conditions, then
\begin{equation}\label{existence.lim.ae}
	\lim\limits_{s\to 0^{+}}  \left(e^{-s\mathcal{L}}f\right)(g) = f(g)  \quad \text{ a.e. } g\in\mathbb{H}_{n}.
\end{equation}
\end{prop}

\begin{proof}
It is enough to show that the desired limit holds a.e. $g\in\mathbb{H}_{n}$ such that $d(g)\le k$ for each $k\in\mathbb{N}$. Let us fix $k\in\mathbb{N}$ and split the initial data $f$ as $$f(h) = f(h)\chi_{\{d(h)\le 2k\}}(h) + f(h)\chi_{\{d(h)> 2k\}}(h) =: f_{1}(h) + f_{2}(h).$$
	
Let us see first that $\lim\limits_{s\to 0^{+}} \left(e^{-s\mathcal{L}}f_{2}\right)(g) = 0$ for every $g\in\mathbb{H}_{n}$ such that $d(g)\le k$. Fix such a point $g$ and let $s>0$. By estimate \eqref{est.phi.q} we have
$$ |\left(e^{-s\mathcal{L}}f_{2}\right)(g)| \le \int\limits_{\mathbb{H}_{n}} q_{s}(gh^{-1}) |f_{2}(h)| dh \le C \int\limits_{d(h)>2k} \phi_{s}(gh^{-1}) |f(h)| dh. $$
	
For $d(h)>2k$ and $d(g)\le k$, the triangle inequality implies $d(gh^{-1}) \ge \frac{d(h)}{2}$. Thus, we can bound the exponential factor as
$$ e^{-\frac{d(gh^{-1})^{2}}{4s}} \le e^{-\frac{d(h)^{2}}{16s}} = e^{-\frac{d(h)^{2}}{32s}} e^{-\frac{d(h)^{2}}{32s}} \le e^{-\frac{k^{2}}{8s}} e^{-\frac{d(h)^{2}}{32s}}. $$
	
Since $d(gh^{-1})\le \frac{3}{2}d(h)$, the polynomial term satisfies
$$ \left( {1 + \frac{d(gh^{-1})^{2}}{s}} \right)^{n-1} \le \left( {1 + 9\frac{d(h)^{2}}{4s}} \right)^{n-1} = \left( {1 + 72\frac{d(h)^{2}}{32s}} \right)^{n-1}. $$
	
Regarding the horizontal variable, if $|z|>2k$, then $|z-z_{0}| \ge \frac{|z|}{2}$ since $|z_0| \le k$. This gives us
$$ \left( {1 + \frac{|z-z_{0}|d(gh^{-1})}{s} } \right)^{-n+\frac{1}{2}} \le \left( {1 + \frac{|z|d(h)}{4s} } \right)^{-n+\frac{1}{2}}.$$

Combining these inequalities, for $h$ in the support of $f_2$ we obtain
$$ \phi_{s}(gh^{-1}) \le C e^{-\frac{k^{2}}{8s}} \phi_{8s}(h). $$

For $s$ sufficiently small, we can bound $\phi_{8s}(h) \le C \phi_{S}(h)$. Since $\int_{\mathbb{H}_n} \phi_S(h)|f(h)|dh < \infty$ by condition (iii) of Proposition \ref{Prop:integrability.conditions}, and $e^{-\frac{k^{2}}{8s}} \to 0$ as $s \to 0^{+}$, it follows that the integral tends to $0$. The case $|z| \le 2k$ follows identically since the horizontal factor is bounded by $1$. Thus, the limit for $f_2$ vanishes.

On the other hand, since $f$ is locally integrable, $f_{1}$ vanishes outside a compact set and belongs to $L^{1}(\mathbb{H}_{n})$. Applying the standard almost everywhere convergence theorem for $L^1$ functions, we deduce
$$ \lim\limits_{s\to 0^{+}}  \left(e^{-s\mathcal{L}}f_ {1}\right)(g) = f(g) \quad \text{ a.e. } g\in\mathbb{H}_{n} \text{ such that } d(g)\le k. $$
This completes the proof.
\end{proof}

\section{Proofs of the Main Results}\label{Sec:proofs.of.thms}

Theorem \ref{Thm:integrability.conditions} follows immediately from Propositions \ref{Prop:integrability.conditions}, \ref{Prop:conv.is.C.infty} and \ref{Prop:limit.holds.ae}.

\begin{proof}[Proof of Corollary \ref{Cor:weight.characterization}]
If a weight $v$ satisfies $v^{-\frac{1}{p}}\phi_{s}\in L^{p'}(\mathbb{H}_{n})$ for every $s\in (0,S)$, then to guarantee that $v$ belongs to the class $D_{p}(\mathcal{L},S)$ we take a function $f\in L^{p}(\mathbb{H}_{n},v)$ and check that the conditions of Theorem \ref{Thm:integrability.conditions} hold. This is achieved by showing that $f$ satisfies the integrability condition against $\phi_s$. Indeed,
\begin{align*}
	\int\limits_{\mathbb{H}_{n}} f(g)\phi_{s}(g)dg = &  \int\limits_{\mathbb{H}_{n}} f(g)v^{\frac{1}{p}}(g)v^{-\frac{1}{p}}(g)\phi_{s}(g)dg \\
	\le & \left(  \int\limits_{\mathbb{H}_{n}} f(g)^pv(g)dg  \right)^{\frac{1}{p}} \left(  \int\limits_{\mathbb{H}_{n}} \left(v^{-\frac{1}{p}}(g)\phi_{s}(g)  \right)^{p'}dg  \right)^{\frac{1}{p'}},
\end{align*}
which is finite by the assumptions on $f$ and $v$. 

Conversely, we use a duality argument based on Landau's resonance Theorem (see \cite{BS}, Lemma 2.6, page 10). Given $s\in(0,S)$ and $v\in D_{p}(\mathcal{L},S)$, to show that $v^{-\frac{1}{p}}\phi_{s}\in L^{p'}(\mathbb{H}_{n})$ it suffices to check that $\int\limits_{\mathbb{H}_{n}} f v^{-\frac{1}{p}}\phi_{s} dg < \infty$ for any nonnegative $f\in L^p(\mathbb{H}_{n})$. Let $f\in L^p(\mathbb{H}_{n})$ with $f \ge 0$ and define $\overline{f}(g)=f(g) v^{-\frac{1}{p}}(g).$ Then $\overline{f}\in L^p(\mathbb{H}_{n},v)$ and 
$$ \int\limits_{\mathbb{H}_{n}} f v^{-\frac{1}{p}}\phi_{s} dg = \int\limits_{\mathbb{H}_{n}} \overline{f}\phi_{s} dg, $$
which is finite since $v$ belongs to the class $D_{p}(\mathcal{L},S)$.
\end{proof}

The proof of Theorem \ref{Thm:local.maximal.operator.weightedLp.boundedness} follows the lines of the existing literature for related settings. We adapt the required estimates to our context.

We use Lemma \ref{bilateral.estimation} and its Corollary \ref{cor:upper.estimation.on.B_g} to obtain the following estimation:
\begin{equation}\label{big.estimation}
	\phi_{s}(gh^{-1}) \le A_2 \phi_{\sigma_2}(h) (\chi_{\mathcal{A}_{g}}(h) + \varphi_{\sigma_2}(g) \chi_{\mathcal{B}_{g}}(h)) + \sup\limits_{h\in B(g,Md(g))} \phi_s(h),
\end{equation}
where $M>1$ is chosen so that $\sigma_2$ lies in the required interval. From \eqref{big.estimation} and the behavior of the heat kernel \eqref{est.phi.q}, it follows that for $a<S$ there exist constants $C_1,C_2>0$ such that
\begin{equation}\label{Qast.bound}
	Q_a^{\ast}f(g) = \sup\limits_{0<s<a} |e^{-s\mathcal{L}}f(g)| \le C_1 (1+\varphi_a(g)) \int\limits_{\mathbb{H}_n} |f(h)| \phi_a(h) dh + C_2 \mathcal{M}_R f(g),
\end{equation}
where $\mathcal{M}_R f(g)$ denotes the centered local Hardy-Littlewood maximal operator defined for $R>0$ by
\begin{equation}
	\mathcal{M}_R f(g) = \sup\limits_{0<r<R} \frac{1}{|B(g,r)|} \int\limits_{B(g,r)} |f(h)|dh.
\end{equation}
The notation $|E|$ for a measurable set $E\subset\mathbb{H}_n$ denotes its volume: $|E| = \int_{\mathbb{H}_n} \chi_E(h) dh$. We use the Carnot-Carathéodory distance to define the balls $B(g,r)$, hence the operator $\mathcal{M}_R$ depends on this definition. Using another distance, such as the Korányi metric, would lead to an equivalent operator, and both are treated similarly in the literature.

The $L^p(\mathbb{H}_n)$ boundedness of the nonlocal Hardy-Littlewood maximal operator follows from the general theory of spaces of homogeneous type. The next Theorem provides the exact two-weight estimate needed for our purpose.

\begin{thm}\label{thm:2.wgt}
Consider $1<p<\infty$ and a weight function $v$ such that $v^{-\frac{1}{p}}\in L^{p'}_{loc}(\mathbb{H}_n)$. Then there exists a weight function $u$ such that 
$$ \mathcal{M}_R : L^p(\mathbb{H}_n,v) \to L^p(\mathbb{H}_n,u) $$
boundedly.
\end{thm}

\begin{proof}
We partition the Heisenberg space by choosing $b>1$ and defining the sets
$$ E_0 := \{g\in\mathbb{H}_n: d(g)<1\}, $$
$$ E_k := \{g\in\mathbb{H}_n: b^{k-1}\le d(g) < b^{k}\}, \quad \text{ for } k\in\mathbb{N}. $$
For each $k\in\mathbb{N}_0$ and $f\in L^p(\mathbb{H}_n,v)$, we decompose the function as $$f = f \chi_{\overline{B(0,Rb^k)}} + f \chi_{\overline{B(0,Rb^k)}^{c}} =: f' + f''.$$ It follows that $\mathcal{M}_R {f''}(g)=0$ for every $g \in E_k$.
	
The existence of the component weights $U_k$ supported on $E_k$ relies on the Rubio de Francia factorization Theorem (which can be found in \cite{GCRDF1985}, Section VI, Theorem 4.2). For a sequence of measurable functions $\textbf{f}=(f_j)_j$ on $\mathbb{H}_n$, let $T \textbf{f} = (T f_j)_j$ for a sublinear operator $T$ and $|\textbf{f}|_p = \left( \sum_j |f_j|^p \right)^{\frac{1}{p}}$. 
	
Choosing $r<1$ and fixing $k\in\mathbb{N}_0$, we write
\begin{equation}\label{v-v.ineq.1}
	\left|\left| \left(\sum\limits_{j} |\mathcal{M}_R f_j|^p \right)^{\frac{1}{p}} \right|\right|_{L^r(E_k)} \le \left|\left| |\mathcal{M}_R \textbf{f'} |_p \right|\right|_{L^r(E_k)},
\end{equation}
where $|\mathcal{M}_R \textbf{f'} |_p = \left(\sum_{j} |\mathcal{M}_R {f_j}'|^p \right)^{\frac{1}{p}}$. Using the distribution function characterization of the $L^r(E_k)$ norm and Kolmogorov's condition \cite{GCRDF1985}, we obtain
\begin{equation}\label{v-v.ineq.2}
	|| |\mathcal{M}_R \textbf{f'} |_p ||_{L^r(E_k)} \le c |E_k|^{\frac{1-r}{r}} || |\mathcal{M} \textbf{f'} |_p ||_{L^1(\mathbb{H}_n)},
\end{equation}
where $c=c_{r,p}$. The Fefferman-Stein vector-valued maximal function inequality for spaces of homogeneous type \cite{GLY2009} implies
\begin{equation}\label{v-v.ineq.3}
	|| |\mathcal{M} \textbf{f'} |_p ||_{L^1(\mathbb{H}_n)} \le c || |\textbf{f}|_p ||_{L^1(\mathbb{H}_n)} = c \left|\left| \left( \sum\limits_j |f_j|^p \right)^{\frac{1}{p}} \right|\right|_{L^1(\mathbb{H}_n)}.
\end{equation}
Applying Hölder's inequality to the product $(f_j v^{\frac{1}{p}})v^{-\frac{1}{p}}$ gives us
\begin{equation}\label{v-v.ineq.4}
	\left|\left| \left( \sum\limits_j |f_j|^p \right)^{\frac{1}{p}} \right|\right|_{L^1(\mathbb{H}_n)} \le V_k \left( \sum\limits_j ||f_j||^p_{L^p(\mathbb{H}_n,v)}\right)^{\frac{1}{p}}, 
\end{equation}
where
\begin{equation}\label{Vk}
	V_k=\left( \int\limits_{\overline{B(0,Rb^k)}} v^{-\frac{p'}{p}}(h)dh \right)^{\frac{1}{p'}},
\end{equation}
which is finite since $v^{-\frac{1}{p}} \in L^{p'}_{loc}(\mathbb{H}_n)$. Combining \eqref{v-v.ineq.1}-\eqref{v-v.ineq.4}, we get the vector-valued estimate
\begin{equation}\label{v-v.ineq}
	\left|\left| \left(\sum\limits_{j} |\mathcal{M}_R f_j|^p \right)^{\frac{1}{p}} \right|\right|_{L^r(E_k)} \le c V_k |E_k|^{\frac{1-r}{r}} \left( \sum\limits_j ||f_j||^p_{L^p(\mathbb{H}_n,v)}\right)^{\frac{1}{p}}.
\end{equation}
	
By the Rubio de Francia factorization Theorem, this estimate for $0<r<1<p<\infty$ guarantees the existence of a weight $U_k$ supported in $E_k$ such that 
\begin{equation}\label{U.norm}
	||U_k^{-1}||_{L^{\frac{r}{p-r}}(\mathbb{H}_n)} \le 1,
\end{equation} 
and for any $f\in L^{p}(\mathbb{H}_n,v)$,
$$ \int\limits_{E_k} |\mathcal{M}_R f(h)|^p U_k(h) dh \le c_k^p ||f||^p_{L^{p}(\mathbb{H}_n,v)}, $$
where $c_k = c V_k |E_k|^{\frac{1-r}{r}}$.

We construct the weight function $u$ by 
\begin{equation}\label{wgt.u}
	u(g) = \sum\limits_{k\in\mathbb{N}_0} (b^k c_k)^{-p} U_k(g)\chi_{E_k}(g).
\end{equation}
Thus, if $f\in L^{p}(\mathbb{H}_n,v)$, we conclude
\begin{equation}\label{mxml.bd.series}
	\int\limits_{\mathbb{H}_n} |\mathcal{M}_{R}f(g)|^p u(g) dg \le \sum\limits_{k\in\mathbb{N}_0} b^{-kp} ||f||^p_{L^p(\mathbb{H}_n,v)},
\end{equation}
and the geometric series converges since we chose $b>1$.
\end{proof}


We need one more observation before we can proceed to prove Theorem \ref{Thm:local.maximal.operator.weightedLp.boundedness}. From Theorem \ref{thm:2.wgt} we can deduce the following Corollary:

\begin{cor}\label{Cor:u.in.Dq}
With the notation of Theorem \ref{thm:2.wgt}, if $q>p$ and $v\in D_{p}(\mathcal{L},S)$ then we can choose the weight $u\in D_{q}(\mathcal{L},S)$.
\end{cor}
\begin{proof}
Let us proceed with the weight function $u$ constructed in Theorem \ref{thm:2.wgt}. We will show that by choosing $b>1$ sufficiently large, the weight $u$ belongs to the desired class.
	
Indeed, from \eqref{wgt.u} the weight function $u$ is of the form $u(g)=\sum_{k\in\mathbb{N}_0} \rho_k \omega_k(g)$, with $\rho_k = (b^k c_k)^{-p}$ and $\omega_k(g)=U_k(g)\chi_{E_k}(g)$, where $c_k=c V_k|E_k|^{\frac{1-r}{r}}$. This gives us the estimation
\begin{align}\label{u.bd.1}
	\int\limits_{\mathbb{H}_n} (u(g)^{-\frac{1}{q}}\phi_s(g))^{q'}dg & \le \sum\limits_{k\in\mathbb{N}_0}  \rho_k^{-\frac{q'}{q}} \int\limits_{E_{k}} U_k(g)^{-\frac{q'}{q}} \phi_s(g)^{q'} dg.
\end{align}
Now we choose $r<1$ such that $\frac{r}{p-r}=\frac{q'}{q}$. Hence, $r=\frac{p}{q}$ and $\frac{1-r}{r}=\frac{q}{p}-1$. Thus, from \eqref{U.norm} and the fact that $\phi_s$ is bounded on each compact block, the integral in the right hand side of \eqref{u.bd.1} is controlled. Next, for the constant part of the series we use \eqref{Vk} to write:
\begin{align*}
	V_k = & || v^{-\frac{1}{p}} \phi_s \phi_s^{-1} \chi_{B(0,Rb^k)}||_{L^{p'}(\mathbb{H}_n)} \le \max_{g\in{B(0,Rb^k)}}\phi_s^{-1}(g) || v^{-\frac{1}{p}} \phi_s||_{L^{p'}(\mathbb{H}_n)}.
\end{align*}
Since $v\in D_p(\mathcal{L},S)$, the norm $||v^{-1/p}\phi_s\|_{L^{p'}}$ is finite. The factor $\max \phi_s^{-1}$ grows polynomially with respect to $b^k$. Thus, we obtain
\begin{align}\label{rhok}
	\rho_k^{-\frac{q'}{q}}\le & c b^{-k\frac{p}{q}q'} V_k^{\frac{p}{q}q'} |E_k|^{\frac{p}{q}(1-q')}.
\end{align}
Recall that $E_k\subset B(0,b^k)$, and since $|B(0,b^k)|\le c b^{(2n+2)k} |B(0,1)|$, where $2n+2$ is the homogeneous degree of $\mathbb{H}_n$ (see \cite{Li2009}). By \eqref{rhok} and this estimation of the measure of $E_k$, the polynomial growth in $b^k$ of the terms $V_k$ and $|E_k|$ is compensated by the geometric decay factor $b^{-k\frac{p}{q}q'}$. Going back to \eqref{u.bd.1}, we can choose $b>1$ large enough such that the series converges, which means that $u\in D_{q}(\mathcal{L},S)$.
\end{proof}


\par We are now, finally, ready to prove Theorem \ref{Thm:local.maximal.operator.weightedLp.boundedness}.

\begin{proof}[Proof of  Theorem \ref{Thm:local.maximal.operator.weightedLp.boundedness}]
	
\par Let $1<p<\infty$, $a>0$, $v\in D_{p}(\mathcal{L},S)$, and $f\in L^{p}(\mathbb{H}_n,v)$. To show that there exists a weight $u$ such that $Q_a^{\ast}f\in L^{p}(\mathbb{H}_n,u)$, we recall from \eqref{Qast.bound} that the operator satisfies the estimate
\begin{align} \nonumber
	Q_a^{\ast}f(g) & = \sup\limits_{0<s<a} |e^{-s\mathcal{L}}f(g)|  \le Af(g) + Bf(g),
\end{align}
where $Af(g) = C_1 (1+\varphi_a(g)) \int\limits_{\mathbb{H}_n} |f(h)| \phi_a(h) dh $, with the function $\varphi_a$ defined in \eqref{Qast.bound} by choosing a suitable constant $M>1$ such that $\varphi_a(g)=M^2 \left( 1 + \frac{|z_0|d(g)}{a} \right)^{n-\frac{1}{2}}$, and $Bf(g)= C_2 \mathcal{M}_R f(g)$.
	
\par We first observe that from H\"older's inequality we have that
\begin{align*}
	|Af(g)| \le & c_1 (1+\varphi_a(g)) ||v^{-\frac{1}{p}}\phi_a||_{L^{p'}(\mathbb{H}_{n})} ||f||_{L^p(\mathbb{H}_n,v)}.
\end{align*}
|
	
\par Next, we apply Theorem \ref{thm:2.wgt} to obtain a weight $u_2(g)$ that satisfies
\begin{align*}
	||Bf(g)||_{L^p(\mathbb{H}_n,u_2)} \le & c ||f||_{L^p(\mathbb{H}_n,v)}.
\end{align*}
And by Corollary \ref{Cor:u.in.Dq}, this weight $u_2$ can chosen the class $D_{q}(\mathcal{L},S)$.
	
\par It remains only to define the weight $u(g):=\min\{u_1(g),u_2(g)\}$, which satisfies all the required properties.
	
\end{proof}

\end{document}